\documentclass[graybox]{svmult}
\usepackage{mathptmx}       
\usepackage{helvet}         
\usepackage{courier}        
\usepackage{hyperref}
\usepackage{graphicx,bm,indentfirst}
\usepackage{amsmath, amssymb,amsfonts}        
\usepackage{subfigure,subeqnarray,latexsym,fancybox,comment}
\usepackage{epsfig,color,datetime}
\usepackage{multirow}
\usepackage{colortbl,array}
\usepackage[table]{xcolor}
\usepackage[numbers, square, comma]{natbib}
\usepackage{algpseudocode,algorithm,algorithmicx}
\usepackage{setspace}
\usepackage{url}
\usepackage{makeidx}         

\title*{$P_1$--nonconforming polyhedral finite elements in high
  dimensions
}
\titlerunning{$P_1$--nonconforming polyhedral finite elements in high dimensions}
\author{Dongwoo Sheen}
\institute{
\at Department of Mathematics, Seoul National
University, Seoul 08826, Korea. \email{sheen@snu.ac.kr.}\\
in press: 
in 2018 MATRIX Annals,
J. de Gier et al. (eds.), MATRIX Book Series 3,
pp. 121--133, 
\copyright Springer Nature Switzerland AG 2020.
\url{https://doi.org/10.1007/978-3-030-38230-8_9}
}
\newcommand{\lemref}[1]{Lemma~\ref{#1}}

\newcommand{\jump}[2]{\left[\left[{#1}\right]\right]_{#2}}

\def\hK{{\hat K}}
\def\hS{{\hat \Delta}}
\def\hQ{{\hat Q}}
\def\hP{{\hat P}}
\def\var{{\varphi}}
\def\Sig{{\Sigma}}
\def\hSig{{\hat\Sigma}}
\def\ol{\overline}
\newcommand{\eq}[1]{\begin{eqnarray}\label{#1}}
\newcommand{\qe}{\end{eqnarray}}

\newcommand{\be}{\begin{eqnarray}}
\newcommand{\ee}{\end{eqnarray}}
\newcommand{\bal}{\begin{aligned}}
\newcommand{\eal}{\end{aligned}}
\newcommand{\bes}{\begin{eqnarray*}}
\newcommand{\ees}{\end{eqnarray*}}
\newcommand{\bs}{\begin{subeqnarray}}
\newcommand{\es}{\end{subeqnarray}}
\newcommand{\bss}{\begin{subeqnarray*}}
\newcommand{\ess}{\end{subeqnarray*}}

\def\suchthat{\,\mid\,}
\def\forany{\,\,\forall}

\newcommand{\csum}[2]{{\sideset{}{^c}\sum_{#1}^{#2}}}
 \def\mbzero{\mathbf{0}}
 \def\lam{\lambda}
 \def\mba{\mathbf{a}}

\def\mbx{\mathbf{x}}
\def\hat{\widehat}
\def\q{\quad}
\def\qq{\qquad}
\def\hQ{\widehat Q}

\def\hxi{\widehat \xi}

\def\hx{\widehat x}

\def\hF{\widehat F}
\def\hve{\hat{\mathbf e}}
\def\bzero{{\mathbf 0}}
\def\hvar{\hat \varphi}
\def\conv{\operatorname{conv}}
\def\ext{\operatorname{ext}}

\def\diam{\operatorname{diam}}

\def\Span{\operatorname{Span}}

\def\cI{\mathcal I}

\def\hv{\widehat v}

\def\mbR{\mathbb R}

\def\tt{T}

\def\p{\partial}

\def\O{\Omega}

\def\NC{\mathcal NC}

\def\NChz{{\mathcal NC}^h_0}
\def\NChz2d{{[\mathcal{NC}}^h_0]^2}
\def\tNChz2d{\widetilde {[\mathcal{NC}}^h_0]^2}

\def\G{\Gamma}

\def\and{\quad\text{and}\quad}

\def\<{\left\langle}
\def\>{\right\rangle}

\def\cM{\mathcal M}

\def\mbc{\mathbf c}

\def\mbA{\mathbf A}

\def\b1{\mathbf 1}
\def\bx{\mathbf x}

\def\Tau{{\mathcal T}}

\def\grad{\nabla\,}
\def\div{\nabla\cdot}

\def\dim{\operatorname{dim}\,}

\def\mbx{\mathbf x}

\newcolumntype{x}[1]{>{\centering\hspace{0pt}}p{#1}} 

\newcommand{\vertiii}[1]{{\left\vert\kern-0.25ex\left\vert\kern-0.25ex\left\vert #1 \right\vert\kern-0.25ex\right\vert\kern-0.25ex\right\vert}}

\begin{document}

\maketitle

\allowdisplaybreaks

\newif\iflong
\longfalse

\begin{abstract} \ We consider the lowest--degree nonconforming finite element
  methods for the approximation of elliptic problems in high dimensions.
The $P_1$--nonconforming polyhedral finite element is introduced for 
any high dimension. Our finite element is simple and cheap as it is based on
the triangulation of domains into parallelotopes,
which are combinatorially equivalent to $d$--dimensional cube,
rather than the triangulation of domains into simplices. Our
nonconforming element is nonparametric, and on each
parallelotope it contains only linear polynomials, but it is sufficient to
give optimal order convergence for second--order elliptic problems.
\end{abstract}

\section{Introduction}
We are interested in the lowest--degree nonconforming finite element
  methods for the approximation of elliptic problems in high
  dimensions.
Efficient numerical methods to approximate solutions of partial
differential equations in high dimensions are very demanding. For
instance, in computational finance, efficient numerical methods are necessary to approximate high
dimensional basket options (see
\cite{broadie2004stochastic, reisinger2007efficient,
  pettersson2008improved} and the references therein). Also, in the
approximation of the Einstein equations of general relativity, one needs to
work on high dimensional dynamics modeling (see
\cite{baumgarte2010numerical,el2005deriving,shapiro1992black}, and the
references therein).
For possible applications in fluid mechanics in high dimensions $\ge
4,$ see \cite{frehse1996existence, frehse1994onregularity,
  frehse1994regularity,frehse1995regular, struwe1995regular} and so
on for the uniqueness, existence and regularity results on the solution of
Navier--Stokes equations. However, practical application areas in
fluid mechanics are hardly found.

In high dimensions it is much simpler to to adopt cubic type of
elements rather than simplicial elements. In our paper we develop
finite elements based on
the triangulation of domains into parallelotopes,
which are combinatorially equivalent to $d$--dimensional cube. In
order to have lowest degree conforming finite elements on $d$--cubes,
one
needs to have multilinear polynomial spaces whose dimensions are at least $2^d.$
Hence to reduce the dimension of approximation polynomial space, we develop
nonconforming elements which are nonparametric, but on each
parallelotope it contains only $P_1$ polynomials which is sufficient to
give optimal order convergence for second--order elliptic problems.

To present most effectively the idea of developing the nonconforming
polyhedral finite elements, which are nonparametric, 
we briefly review the nonconforming elements of lowest degrees from
parametric elements to nonparametric elements,
and from rotated bilinear elements to $P_1$--nonconforming
quadrilateral elements. By this brief review it will be very natural
to expose our idea to develop the final nonconforming polyhedral
elements in high dimensions.

In this section we present our model problem, and then some
  notations and preliminaries are given.
  
  \subsection{The model problem}
  Let $\O\in \mbR^d$ be a simply--connected polyhedral domain with boundary $\G.$
Consider the second--order elliptic problem:
\begin{subeqnarray}\label{eq:ell-diri}
- \div \left(\mathbf A(\bx) \grad u\right) + c u &=& f, \quad \Omega,\\
u &= &0,\quad \G,
\end{subeqnarray}
where the uniformly positive--definite matrix--valued function $\mathbf A$
and the nonnegative function $c>0$ are assumed to be sufficiently smooth.
The weak formulation of \eqref{eq:ell-diri} is to find $u\in
H^1_{0}(\O)$ fulfilling
\begin{eqnarray}\label{eq:diri-weak}
  a(u,v) = \ell(v)\forany v\in H^1_{0}(\O),
\end{eqnarray}
where the bilinear form $a(\cdot,\cdot): H^1_{0}(\O)\times H^1_{0}(\O)\to
\mbR$ and the linear form $\ell: H^1_{0}(\O) \to \mbR$ are given by
\begin{subeqnarray}\label{eq:a-def}
  a(u,v) &=& (\mbA\grad u, \grad v) + (cu,v),\\
  \ell(v) &=& (f,v),
\end{subeqnarray}
for all $u,v\in H^1_0(\O).$

\subsection{Notations and preliminary results}
For be a domain $S\in\mbR^d,$ 
we adopt standard notations for Sobolev spaces with the inner products and norms
\begin{eqnarray*} 
&&L^2(S) = \{f:S\rightarrow\mathbb R \suchthat~ \int_S |f(\bx)|^2 \, dx < \infty\}, \\
&&\qquad(f, g)_S = \int_S f(x) g(x) \, dx;~ \|f\|_{0,S} = \sqrt{(f,f)}; \\
&&H^1(S) = \{f \in L^2(S) \suchthat~ \|\grad f(x)\|_{0,S}  < \infty\}, \\
&&\qquad\qquad (f, g)_{H^1(S)} = (f, g)_S + (\grad f, \grad g)_S;~ \|f\|_{1,S} =
\sqrt{(f,f)_{H^1(S)}};\\
&& H^1_0(S) = \{f \in H^1(S) \suchthat~ f\mid_{\p S}  = 0\};\\
&& H^k(S) = \{f \in L^2(S) \suchthat~ \| \p^\alpha f(\bx)\|_{0,S}  < \infty  \forany
  |\alpha|\le k  \}, \\
&&\qquad (f, g)_{H^k(S)} =\sum_{|\alpha| \le k}  (\p^\alpha f,
  \p^\alpha g)_S;~ \|f\|_{k,S} =
\sqrt{(f,f)_{H^k(S)}}.
\end{eqnarray*}
Here, and in what follows, if $S=\O$ the subindex $\O$ may be
dropped as well as the subindex $0.$

Denote by $\conv S$ the interior of the convex hull of $S,$ which is
an open set.
The $0$-- and $1$--faces of $d$--polyhedral domain $S$ are the vertices and
edges of $S,$ respectively.
In particular, the $(d-1)$--faces of $S$ will be called the
``facets'' of $d$--dimensional polyhedral domain, and by $\mu_j$ we
designate the barycenter of facet $F_j$'s.

\vspace{5mm}
The organization of the paper is as follows. In Section 2, the lowest--degree
parametric and nonparametric nonconforming quadrilateral elements for two and three
dimensions are briefly reviewed.
In Section 3, we introduce the nonparametric $P_1$--NC polyhedral finite
element space in $\mbR^d$ for any $d\ge 2.$ Here, and in what follows,
$P_1$ means ``piecewise linear'' and NC means ``nonconforming.''

\section{The parametric and nonparametric $P_1$--simplicial and quadrilateral nonconforming finite elements}
In this section we review the simplicial and
quadrilateral NC (nonconforming) finite element spaces in two and three
dimensions.

\subsection{The parametric simplicial and rectangular NC elements in two
  and three dimensions }
The NC elements for elliptic and Stokes equations in two
and three dimensions have been well known since the work of Crouzeix
and Raviart \cite{crouzeix-raviart} was published.

Denote the reference element as follows:
\begin{eqnarray}
  \hK =\begin{cases} \hS^d=
d\text{--simplex, \it{i.e.,} }\conv\{\bzero,\hve_1,\cdots, \hve_d\},\\
 \hQ^d=d\text{--cube, \it{i.e.,} } (-1,1)^d.
\end{cases}  
\end{eqnarray}

\begin{enumerate}
\item The lowest--degree simplicial Crouzeix-Raviart element (1973) \cite{crouzeix-raviart}:
  \begin{enumerate}
\item  $\hK=\hS^d,\, d=2,3;$
\item $\hP_\hK = P_1(\hK)=\Span \{1, \hx_1, \cdots, \hx_d\};$
\item $\hSig_\hK=\{\hvar(\hxi_j),\quad \hxi_j \mbox{ barycenter of facets},j=1,\cdots,d+1,\forany \hvar\in\hP(\hK)\}.$
\end{enumerate}
All odd--degree simplicial NC elements were introduced for
the Stokes problems in \cite{crouzeix-raviart}.

 \begin{remark} It is straightforward to define the simplicial NC
   elements on $d$--simplicial triangulation in any high
   dimension. However, for high dimension it is not easy to see how the
   $d$--simplices
   are packed in the domain. Thus the development of
   $d$--cubical elements is beneficial in this regard.
 \end{remark}
 
\item The Han rectangular element (1984) \cite{han84}:
    \begin{enumerate}
\item $\hK = \hQ^2;$
\item $\hP_\hK = P_1(\hK)\oplus \Span \{\hx_1^2-\frac53\hx_1^4,\hx_2^2-\frac53\hx_2^4\};$
\item $\hSig_\hK=\{\hvar(\hxi_j), \hxi_j, j=1,\cdots,4, \mbox{ midpoints of facets};
\int_{\hQ^2}\hvar \forany \hvar\in\hP_\hK\}.$
\end{enumerate}

\item The Rannacher--Turek rotated $Q_1$ element (1992, 
  \cite{rannacher-turek}, also Z. Chen \cite{chen-projection-93}):
  \begin{enumerate}
\item  $\hK = \hQ^d, d=2,3;$
\item $\hP_\hK = P_1(\hK)\oplus \Span \{\hx_1^2-\hx_d^2, \hx_{d-1}^2-\hx_d^2\};$
\item $\hSig_\hK^{(m)}= \{\hvar(\hxi_j),\q \hxi_j, j=1,\cdots,2d,
  \mbox{barycenters of facets } \hF_j, \forany\hvar\in\hP_\hK\};$ \newline
$\hSig_\hK^{(i)} = \{\frac1{|\hF_j|}\int_{\hF_j}\hvar d\sigma,\q \hF_j,j=1,\cdot,2d,
\mbox{are facets}, \forany\hvar\in\hP_\hK\}.$
\end{enumerate}
\begin{remark} The two DOFs generate two different NC elements, and
  for general quadrilateral meshes the NC element with the DOFs
  $\hSig_\hK^{(i)}$ gives optimal convergence rates while that with the DOFs
  $\hSig_\hK^{(m)}$ leads to suboptimal convergence rates.
\end{remark}

\item The DSSY element(DOUGLAS-SANTOS-Sheen-YE, 1999) \cite{dssy-nc-ell}:
  For $\ell=1,2,$ define
    \begin{eqnarray*}
 \theta_\ell(t) =\left\{
\begin{array}{ll}
    t^2, & \hbox{$\ell=0$;} \\
    t^2-\frac53t^4, & \hbox{$\ell=1$;} \\
    t^2-\frac{25}{6}t^4+\frac72t^6, & \hbox{$\ell=2$.}
  \end{array}
\right.
\end{eqnarray*}

    \begin{enumerate}
\item $\hK=\hQ^d, d=2,3;$
\item $\hP_\hK = P_1(\hK)\oplus \Span
  \{\theta_\ell(\hx_1)-\theta_\ell(\hx_d), \theta_\ell(\hx_{d-1})-\theta_\ell(\hx_d) \};$
  \item $\hSig_\hK^{(m)} = \{\hvar(\hxi_j),\hxi_j \mbox{ barycenters of
      facets}, j=1,\cdots,2d, \forany\hvar\in\hP_\hK\}$ \newline
    $\hSig_\hK^{(i)}=\{\frac1{|\hF_j|}\int_{\hF_j}\hvar d\sigma ,\, \hF_j, j=1,\cdots,2d, \mbox{ are facets}, \forany\hvar\in\hP_\hK\}.$
  \end{enumerate}
  \begin{remark} The benefit of the DSSY element is the Mean Value Property
    \begin{eqnarray}\label{eq:mvp}
      \hvar(\hxi_j) = \frac1{|\hF_j|}\int_{\hF_j}\hvar d\sigma
      \forany \hvar\in \hP_\hK
    \end{eqnarray}
    holds if $\ell=1,2.$ Thus, for $\ell=1,2,$ 
    the two DOFs $\hSig_\hK^{(m)}$ and
  $\hSig_\hK^{(i)}$ generate an identical NC elements
with optimal convergence rates. The case of $\ell=0$ reduces to
the Rannacher--Turek rotated $Q_1$ element.
\end{remark}

\item For truly quadrilateral triangulations,
  $P_1(\hK)$ for the Rannacher--Turek element and the DSSY element
  should be modified such that $P_1(\hK)$ is replaced by
  $Q_1(\hK)$ in
the reference elements with an additional DOF
$\int_{\hQ^2}\hvar(\hx_1,\hx_2)\hx_1\hx_2d\hx_1d\hx_2$
(Cai--Douglas--Santos--Sheen--Ye, CALCOLO, 2000) \cite{cdssy}.

 \end{enumerate}

Let $(\Tau_h)_{0<h<1}$ denote a family of quasiregular triangulations
of $\O$ into simplices or quadrilaterals $K_j$'s where $\diam(K_j)\leq h
\forany K_j\in \Tau_h$. 
If $K$ is a $d$--simplex, or a parallelogram or a parallelepiped,
there is a unique (up to rotation in the order of the vertices)
affine map $F_K\,:\,\hK\to K.$
Set
\[
\NC_K = \{v\,:\,v=\hv\circ F_K^{-1},\ \hv\in \hP_\hK\}.
\]
The global ({\it parametric}) NC element space is defined as follows:
\begin{eqnarray*} 
\NC^h&=&\big\{v\in L^2(\O)\suchthat v|_{K}\in\NC_K \forany K\in \Tau_h;
  \,\<[[v]]_F, 1\>_F = 0 \\
 &&\qq\qq\qq\qq\qq\qq \forany \text{ interior facets } F\in \Tau_h \big\},
\end{eqnarray*}
and 
\[ 
\NC^h_0=\left\{v\in \NC^h\suchthat 
\,\<v_F, 1\>_F = 0 \forany \text{ boundary facets } F\in \Tau_h \right\},
\]
where
$[[v]]_F$ denotes the jump across the facets $F=\p K\cap \p K'$ for
all $K, K'\in \Tau_h.$

The (parametric) NC Galerkin method for \eqref{eq:diri-weak} is to find $u_h\in \NC^h_0$ such that
\begin{eqnarray}
  a_h(u_h,v_h) =\ell(v_h)\forany v_h\in \NC^h_0,
\end{eqnarray}
where 
\[
  a_h(u,v) = \sum_{K\in\Tau_h}
  (\mbA\grad u, \grad v)_K + (cu,v) \forany u,v\in \NC^h_0+H^1_0(\O).
\]

\subsection{The nonparametric  NC quadrilateral and hexahedral elements}
Recall that finite elements need to contain at least the $P_1$ space in order to have
a full approximation property for the second--order elliptic problems
due to the Bramble--Hilbert lemma.

In this subsection the nonparametric DSSY-type nonconforming
quadrilateral elements will be briefly reviewed. Then
  the $P_1$--NC quadrilateral elements will be reviewed, which are
essentially nonparametric, but which are the lowest
degrees--of--freedom elements as they contain only $P_1$ spaces on
each quadrilateral or hexahedron.

\subsubsection{The  nonparametric DSSY-type nonconforming
  quadrilateral elements }
It was questionable if, for truly quadrilateral triangulations, any
4--DOF DSSY--type nonconforming element can be defined or
not. A DSSY--type element needs to fulfill the Mean Value Property
\eqref{eq:mvp} such that $\hSig_\hK^{(m)}$ and $\hSig_\hK^{(i)}$ generate an identical NC elements. It turns
out that we may not have such a finite element in the class of
parametric finite elements. Instead, it is possible to define such
DSSY--type
element in the class of nonparametric finite elements. Indeed, 
a class of nonparametric {DSSY} nonconforming quadrilateral elements
\cite{jeon-nam-sheen-shim-nonpara} were developed with 4 DOFs
fulfilling the Mean Value Property \eqref{eq:mvp}.

Such nonparametric {DSSY} nonconforming hexahedral elements in three
dimensions with 6 DOFs fulfilling three--dimensional Mean Value
Property will appear elsewhere \cite{sheen-shim-3dnonpara}.

\subsubsection{The $P_1$--NC quadrilateral element}
For general convex quadrilateral triangulation ($d=2$ or $d=3$), it is possible to
define a {\it nonparametric} $P_1$--NC quadrilateral element (see 
Park (PhD Thesis, 2002) and Park--Sheen (SINUM, 2003)
\cite{cpark-thesis,  parksheen-p1quad}). 

\begin{enumerate}
  \item The nonparametric $P_1$--NC quadrilateral ($d=2$) or
    hexahedral ($d=3$) element.
\begin{enumerate}
\item $K,$ any convex quadrilateral or parallelopiped;
\item $P_K = P_1(K);$
\item $\Sig_K=\{\var(\mu_j),j=1,\cdots, d +1,\,\forall \var\in P_K\},$   where 
  $\mu_j$ is any barycenter of the two
opposite facets  $F_{j,\pm}$ for  $j=1,\cdots,d,$
 and $\mu_{d+1}$ is any  other barycenter
of facets  $F_{j,\pm}, j=1,\cdots,d.$
\end{enumerate}
\item
\begin{lemma} \cite{cpark-thesis, parksheen-p1quad}.
    If  $u\in P_1(K)$, then $u(\mu_{1,-}) + u(\mu_{1,+}) =\cdots =
    u(\mu_{d,-}) + u(\mu_{d,+}) $.
    Conversely, if $u_{j,\pm}$ are given values at $\mu_{j,\pm}$, for $1 \leq j
    \leq d$, satisfying $u_{1,-} + u_{1,+} = \cdots= u_{d,-} + u_{d,+}$, 
    then there exists a unique function $u \in P_1(K)$ such that 
    $ u(\mu_{j,\pm}) = u_{j,\pm}$, $1 \leq  j \leq d.$
  \end{lemma}
\end{enumerate}

\begin{figure}
\begin{center}
\includegraphics[height=4.0cm,width=5.cm,angle=0,clip=]{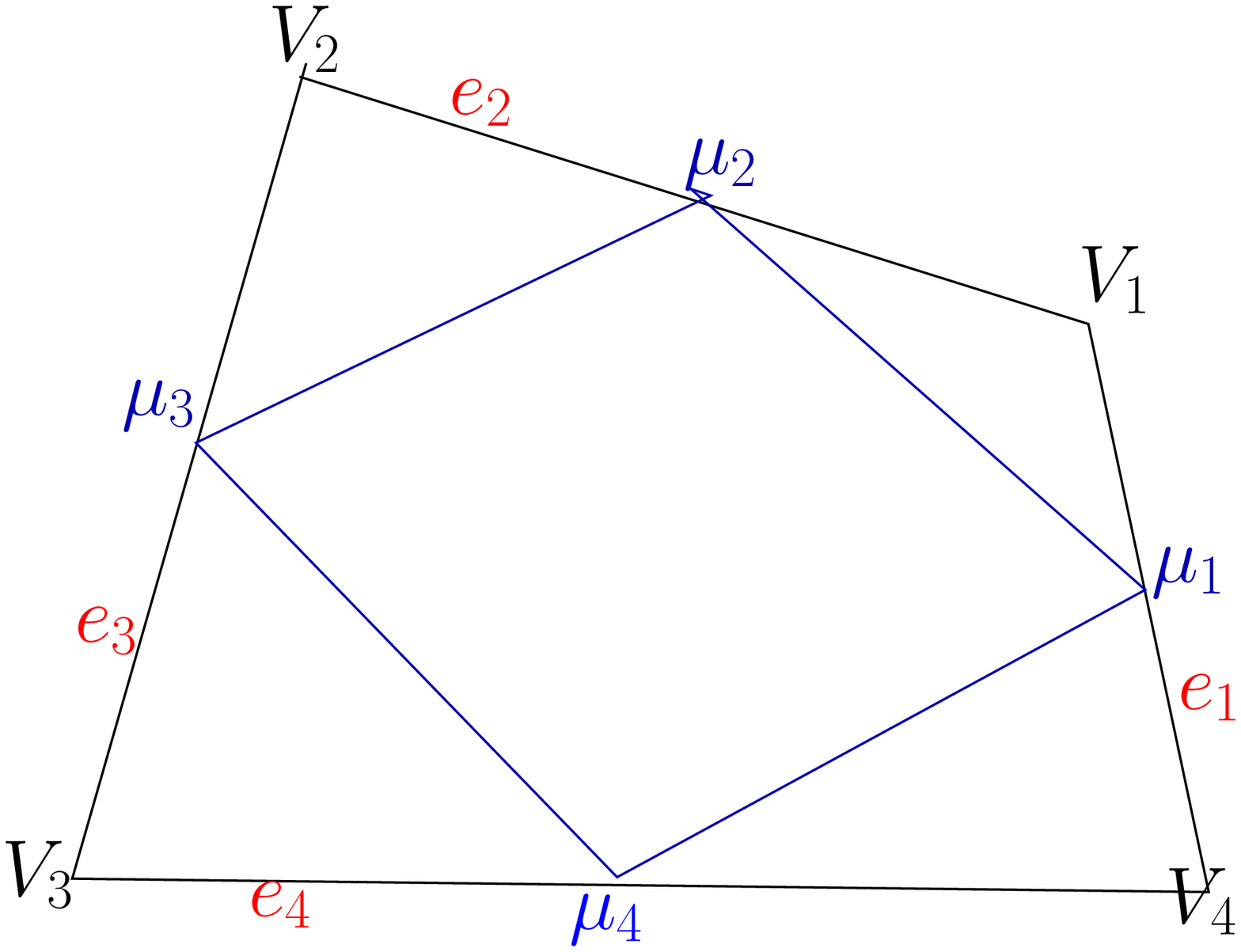}     
\includegraphics[height=4.0cm,width=5.cm,angle=0,clip=]{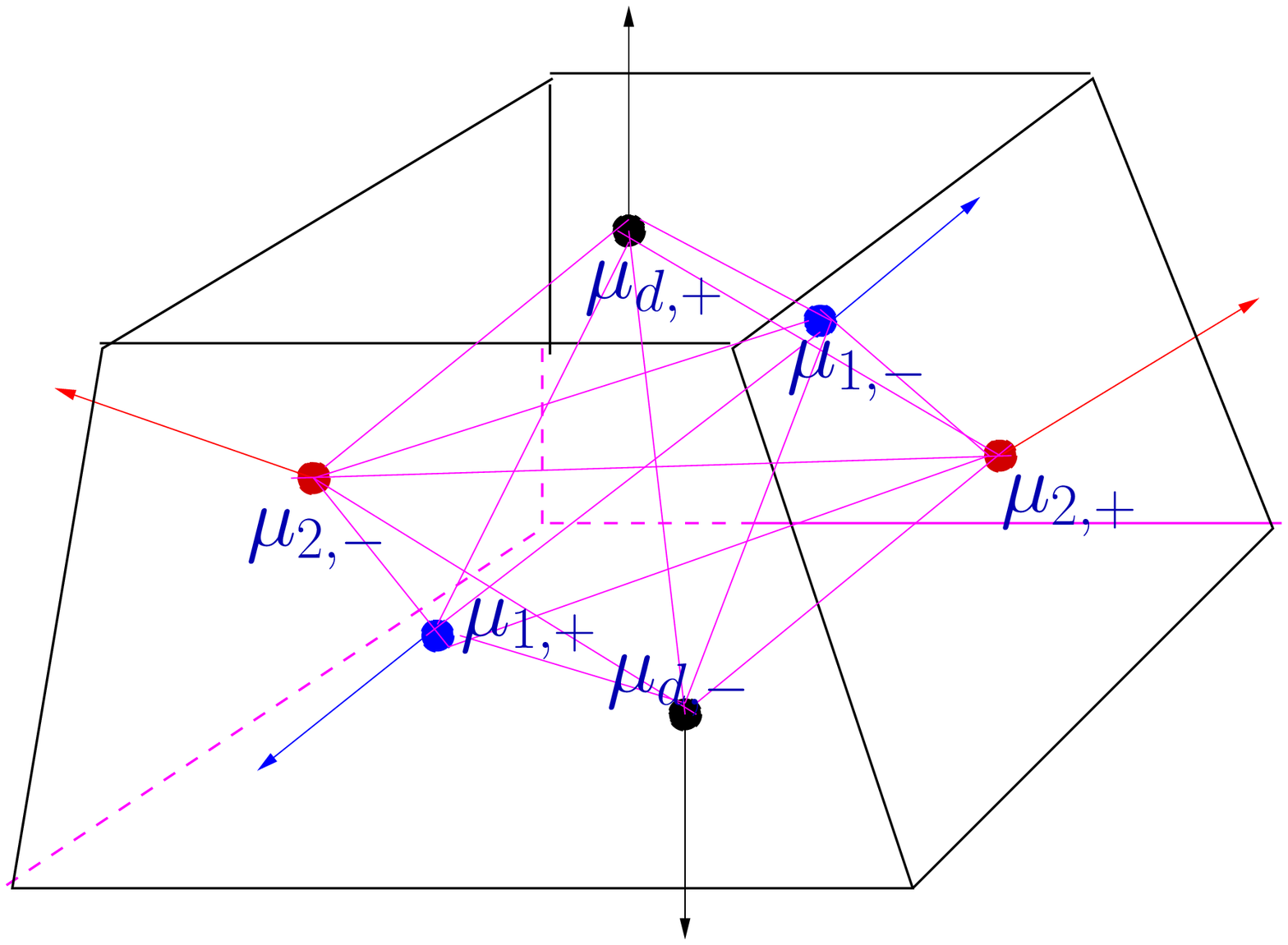}
\caption{{\bf Left.} For $j=1,\cdots,4,$ $\mu_j$ denotes the midpoint of
  edge $e_j$ of any quadrilateral $\conv\{V_1,V_2,V_3,V_4\}.$
 Then $\conv\{\mu_1,\mu_2,\mu_3,\mu_4\}$ is a parallelogram.
{\bf Right.} For $j=1,2,3,\iota=\pm,$ $\mu_{j,\iota}$ denotes the barycenter of
  face $f_{j,\iota}$ of any hexahedron.
 Then $\conv\{\mu_{j,\iota},j=1,2,3,\iota=\pm\}$ forms an octahedron,
 which is a dual of the hexahedron.}
\end{center}
\end{figure}
It is shown in \cite{cpark-thesis,  parksheen-p1quad} that
the above elements are unisolvent and optimal error estimates hold for
the second--order elliptic problems \eqref{eq:diri-weak}.

\section{The $P_1$--nonconforming polyhedral finite element} 
We now extend the $P_1$--NC quadrilateral or hexahedral
element to any dimension $d\ge 2.$

The notion of polytope is the generalization of
quadrilateral to higher dimension, introduced by Coexter \cite{coxeter2012regular}.
See also \cite{brondsted2012introduction, grunbaum1967convex}.
The stream of developing the $P_1$--NC polyhedral finite element basis
functions is a follow--up of that given in \cite{cpark-thesis, parksheen-p1quad}.

For polyhedral set, polytope, parallelotope, and so on, we adopt the
following definitions. Here, we just modify to have those sets to be open
sets instead of closed sets.

\begin{definition}\cite[p.26]{grunbaum1967convex}
A set $K \subset \mathbb R^d$ is called a {\it polyhedral} set provided $K$ is the intersection of a
finite family of open half spaces of $\mathbb R^d.$
\end{definition}

\begin{definition}\cite[p.17, p.31]{grunbaum1967convex}
Let $K$ be a convex subset of $\mathbb R^d.$ A point $x \in \ol K$ is an extreme point of $K$
provided $y, z\in  \ol K, 0 < \lam < 1,$ and $x = \lam y + (1-
\lam)z$ imply $x = y = z.$
The set of all extreme points of $K$ is denoted by $\ext K.$ 
An open convex set $K \subset \mathbb R^d$ is a {\it polytope} provided
$\ext K$ is a finite set.
For a polytope of dimension $d$, we use $d$--polytope.
We use $k$--face if the face is of dimension $k$.
A subset $F \subset \ol K$ is called a {\it face} of a polytope $K$ if either
$F = \emptyset$ or $F = K,$ or if there exists a supporting hyperplane $H$ of $K$
such that $F = \ol K \cap H.$
The set of all faces of $K$ is denoted by $\mathcal F (K)$.
The $0$-- and $1$--faces of $d$--polytope $K$ are the vertices and
edges of $K,$ respectively.
In particular, the $(d-1)$--faces of $d$--polytope $K$ will be called the
{\it facets} of $K.$
For a polytope (or polyhedral set) $K,$  $\ext K$ consists of all
vertices of $K.$
\end{definition}

The following proposition is a well--known result from the above definitions.
\begin{proposition}
A set $K \subset \mathbb R^d$ is a polytope
 if and only if $K$ is a bounded polyhedral set.
\end{proposition}

 \begin{definition} \cite{brondsted2012introduction}
 We say
 $\sum_{j=1}^d \lam_j\mbx_j$
 is a {\it convex combination}
 of $\mbx_j\in\mbR^d, j=  1,\cdots,d,$
   denoted by
  \begin{eqnarray}\label{eq:csum}
       \csum{j=1}{d} \lam_j\mbx_j.
  \end{eqnarray}
 if $\sum_{j=1}^d\lam_j = 1$ and $\lam_j\ge 0\,\,\forall j.$
  The vectors $\mbx_j\in\mbR^d, j=  1,\cdots,d,$ are said to be
  {\it affinely independent} if
  \[
    \sum_{j=1}^d \lam_j\mbx_j =\mbzero\quad\text{with }
      \sum_{j=1}^d\lam_j=0 \text{ implies } \lam_j = 0\forany j.
    \]
  \end{definition}

For affinely independent vectors $\mbx_j, j = 1, \cdots, k,$
a {\it $k$--parallelotope} $K$
is a bounded polytope which can be
represented by
\begin{eqnarray}
  \mbx = \mba + \sum_{j=1}^k \lam_j\mbx_j,\,
  0\le \lam_j\le 1\forany j.
\end{eqnarray}
In the meanwhile a bounded $k$--polytope can be represented by
  \begin{eqnarray*}
  \mbx = \csum{j=1}{2k} \lam_j\mbx_j.
  \end{eqnarray*}
  with suitable $\mbx_j\in \ext{K}, j = 1, \cdots, 2k,$
if it is combinatorially equivalent to a $k$--cube.
Two polytopes are said to be {\it combinatorially equivalent} if there is a
one--to--one correspondence between the set of all faces of $P$ and that
of all faces of $Q$ with incidence--relation preserved.

If a $k$--polytope is combinatorially equivalent to $k$--cube, $(-1,1)^k,$
$K$ is assumed to have $2d$ boundaries which are flat
$(d-1)$--faces combinatorially equivalent to the $(d-1)$--dimensional cube $(-1,1)^{d-1}.$
In particular, denote by $(F_{j,-},F_{j,+}), j = 1, \cdots, d,$
the pairs of opposite  $(d-1)$-faces.
For each vertex $V_j,$ there are $d$ edges which meet at the vertex.
For $j=1,\cdots,2d,$ denote by $\mu_{j,\pm}$ the barycenter of facet
$F_{j,\pm}.$

The convex hull of the barycenters of facets of $d$--polytope $K$
forms the dual of $K$, and their diagonals
intersect at one point and are bisected by this point. Indeed, we have
the following lemma.
\begin{lemma}\label{lem:poly}
Let $K\in \mathbb R^d$ be a $d$--parallelotope
which is combinatorially equivalent to the $d$--dimensional cube $(-1,1)^d,$
with $2^d$ vertices: $V_{j}, j = 1,\cdots, 2^d.$
Assume that $K$ has $d$ pairs of opposite boundaries
$F_{j,\pm}, j= 1,\cdots,d,$ which are flat $(d-1)$--faces combinatorially equivalent
to the $(d-1)$--dimensional cube $(-1,1)^{d-1}.$
  Let  $\{\mu_{j,+},\mu_{j,-},j=1,\cdots,d\}$
be  the barycenters of boundaries
of $F_{j,\pm}.$ \\
Then $\conv\{\mu_{j,+},\mu_{j,-},j=1,\cdots,d\}$ forms a
$d$--parallelotope, which is the dual of $K,$ and the midpoint of $\mu_{j,+}$ and $\mu_{j,-}$
coincides for $j=1,\cdots,d.$
\end{lemma}
\begin{proof}
For $j= 1,\cdots,d,$ and $\iota=\pm,$ let $V_k^{(j,\iota)},k=1,\cdots,2^{d-1},$
denote the vertices
of $F_{j,\pm}.$ Then notice that
$$\frac12\left[\sum_{k=1}^{2^{d-1}} V_k^{(j,+)} + \sum_{k=1}^{2^{d-1}} V_k^{(j,-)}\right]
= \frac12\sum_{k=1}^{2^d} V_k$$ 
which implies that the midpoint of $\mu_{j,+}$ and $\mu_{j,-}$
coincides
for every $j=1,\cdots,d.$ 
This proves the lemma.
\end{proof}

The \lemref{lem:poly} enables to generalize the
$P_1$--NC quadrilateral or hexahedral element to any $d\ge 2$ dimension.

From now on, {\it we assume that a $k$--polytope is combinatorially equivalent
to a $k$--cube, for $0<k\le d$.} We are ready to generalize the $P_1$--NC quadrilateral element
to any high dimension as follows.

\begin{definition}\label{fem:polytope}
  Define the $d$--dimensional $P_1$--NC polyhedral element
  as follows:
\begin{subeqnarray*}
&\text{(i)}& K, \text{$d$--parallelotope};\\
&\text{(ii)}& P_K = P_1(K);\\
&\text{(iii)}& \Sig_K=\{\var(\mu_j),j=1,\cdots, d +1,\,\forall \var\in P_K\},  \text{ where }
  \mu_j \mbox{ is any barycenter of the two}\\
&&\qquad\text{opposite facets }  F_{j,\pm}\text{ for }  j=1,\cdots,d,
  \text{ and } \mu_{d+1} \text{ is any
    other barycenter}\\
  &&\qquad \text{of facets } F_{j,\pm}, j=1,\cdots,d.
\end{subeqnarray*}
\end{definition}
Now, we have the following lemma.
\begin{lemma}\label{lem:dice-poly}
    If  $u\in P_1(K)$, then the following $d-1$ constraints hold:
$u(\mu_{1,-}) + u(\mu_{1,+}) =\cdots=u(\mu_{j,-}) + u(\mu_{j,+}) =\cdots =u(\mu_{d,-}) + u(\mu_{d,+}).$
    Conversely, if $u_{j,\pm}$ are given values at $\mu_{j,\pm}$, for $1 \leq j
    \leq d$, satisfying 
$u_{1,-} + u_{1,+}=\cdots = u_{j,-} + u_{j,+} =\cdots=u_{d,-} + u_{d,+},$
    then there exists a unique function $u \in P_1(K)$ such that 
    $ u(\mu_{j,\pm}) = u_{j,\pm}$, $1 \leq  j \leq d.$
  \end{lemma}

  \begin{proof} 
Due to \lemref{lem:poly}, we have $\mu_{j,-} + \mu_{j,+} = 2\mbc,
\forany j=1,\cdots,d,$ and the
linearity of $\phi$ implies
$\phi(\mu_{j,-}) + \phi(\mu_{j,+}) = 2\phi(\mbc), \forany j=1,\cdots,d.$

To show the converse suppose that $u_{j,\pm}$ are given values at $\mu_{j,\pm}$, for $1 \leq j
    \leq d$, satisfying 
$u_{1,-} + u_{1,+}=\cdots = u_{j,-} + u_{j,+} =\cdots=u_{d,-} + u_{d,+},$
     Without loss of generality, we may assume that $\mu_j =
\mu_{j,-}$ is chosen from the pair of barycenters $\mu_{j,-}$ and
$\mu_{j,+}$ for all $j=1,\cdots,d.$ 
Since $\conv\{\mbc, \mu_j,\, j=1,\cdots, d\}$ forms a
$d$--simplex, any function $\phi\in P_1(\conv\{\mbc, \mu_j,\,
j=1,\cdots, d\})$ is uniquely determined by the $d+1$ values at 
$\mbc, \mu_j,\,j=1,\cdots, d.$
From the constraint and \lemref{lem:poly}, the value at $\mbc$ can be determined by any
additional
value at any $\mu_{j_0,+}.$ This shows the claim of the converse holds.
  \end{proof}

Owing to \lemref{lem:poly} and \lemref{lem:dice-poly}, the following
unisolvency holds.
\begin{theorem}\label{thm:unisolv}
  The $d$--dimensional $P_1$--NC polyhedral element defined in Definition \ref{fem:polytope} is unisolvent.
\end{theorem}

\subsection{Global  $P_1$--NC polyhedral finite element
  spaces}

Let $(\Tau_h)_{0<h<1}$ denote a family of quasiregular triangulations
of $\O$ into $d$--parallelotopes $K_j$'s where $\diam(K_j)\leq h
\forany K_j\in \Tau_h$ with all their $k$--faces are combinatorially
equivalent to $k$--cube for all $k\le d-1.$
Set
\[
\NC_K = P_1(K)\forany K\in \Tau_h.
\]
The above \lemref{lem:dice-poly} enables to define the $d$--dimensional
  $P_1$--NC polyhedral element spaces, which are nonparametric.
Indeed, the global $P_1$--NC polyhedral finite element spaces are
defined as follows:
\begin{subeqnarray*} 
\NC^h&=&\big\{v\in L^2(\O)\suchthat v|_{K}\in\NC_K \forany K\in \Tau_h;
  \,\<\jump{v}{F}, 1\>_F = 0\\
&&\qquad\qquad \forany \text{ interior } (d-1) \text{--faces (or facets) } F\in \Tau_h \big\},
\end{subeqnarray*}
and 
\[ 
\NC^h_0=\left\{v\in \NC^h\suchthat 
\,\<v_f, 1\>_F = 0 \forany \text{ boundary facets } F\in \Tau_h \right\},
\]
where $\jump{v}{F}$
denotes the jump across the facets $F=\p K\cap \p K'$ for
all $K, K'\in \Tau_h.$

\subsection{Basis  and its dimension}
Following the idea in \cite{cpark-thesis,
  parksheen-p1quad} for two and three dimensions,
denote by $\cM_h$ the set of all
barycenters of facets in $\Tau_h.$
Let $\{V_j\in\Tau_h ,j = 1,\cdots, N_V^i\}$ be the set of all 
interior vertices in $\Tau_h.$ Then for $j=1,\cdots, N_V^i,$
let $K^{(j)}_{l}, l = 1,\cdots,N_j$ form the set of
all $d$--parallelotopes in $\Tau_h$ which share the common vertex $V_j.$
Denote by $\cM(V_j)$ the set of all barycenters of the facets of those
$K^{(j)}_{l}, l = 1,\cdots,N_j$ sharing the common vertex $V_j.$
Now, define $\phi_j\in \NC^h_0$ by
\begin{eqnarray*}
\phi_j(\mu) = \begin{cases} 1,&\quad \mu \in \cM(V_j), \\
  0, & \quad \mu \in \cM_h \setminus \cM(V_j).
\end{cases}
\end{eqnarray*}
Then the following theorem holds (see \cite{cpark-thesis,
  parksheen-p1quad} for two and three dimensions):
\begin{theorem} $\phi_j, j = 1,\cdots, N_V^i$ are linearly
  independent. Moreover, we have
\begin{eqnarray*}
\NC_0^h = \Span\{\phi_j, j = 1,\cdots, N_V^i\}; \quad \dim(\NC_0^h) = N_V^i.
\end{eqnarray*}
\end{theorem}

\subsection{Local and global Interpolation operators}
Let $K$ be a $d$--parallelotope combinatorially equivalent to $[-1,1]^d$ with facets and
barycenters $F_{j}$ and $\mu_j,$ respectively, for $j = 1,\cdots, 2d.$
Denote by $V^{(j)}_k, k = 1,\cdots, 2^{d-1},$ the vertices of $F_j, j =
1,\cdots, 2d.$
Then the interpolation operator $\cI_K: C^0(\overline K)\to P_1(K)$ is
defined as follows:
if $u\in C^0(\overline K),$
due to \lemref{lem:dice-poly} one can define $\cI_K u\in P_1(K)$ such that
\begin{eqnarray*}
  (\cI_K u)(\mu_j) = \frac1{2^{d-1}}\sum_{k=1}^{2^{d-1}}
  u(V^{(j)}_k),\, j = 1,\cdots,2d.
\end{eqnarray*}

The global interpolation operator $\cI_h: C^0(\overline \O)\to \NC_h$
is then defined element by element such that
\begin{eqnarray*}
\cI_h\mid_K = \cI_K\forany K\in \Tau_h.
\end{eqnarray*}
Since linear polynomials remain unchanged by $\cI_h,$
the Bramble--Hilbert lemma (which holds for high dimensional
spaces) leads to the following estimate:
\begin{eqnarray}\label{eq:interpol}
\|\cI_h - u\| + h|\cI_h - u|_{1,h} \le C h^2 |u|_2\forany u \in H^2(\O),
\end{eqnarray}
where $|\cdot|_{1,h}$ designates the broken semi-norm defined by
$|v|_{1,h} = \sqrt{\sum_{K\in\Tau_h} \|\grad v\|_{0,K}^2}$ for all
$v\in H^1(\O)+\NC_h.$

\subsection{The $P_1$--NC  polyhedral Galerkin methods}

Then the NC Galerkin method for \eqref{eq:diri-weak} is to find $u_h\in \NC^h_0$ such that
\begin{eqnarray}\label{eq:nc-ell}
  a_h(u_h,v_h) =\ell(v_h)\forany v_h\in \NC_0^h,
\end{eqnarray}
where 
\[
  a_h(u,v) = \sum_{K\in\Tau_h}
  (\mbA\grad u, \grad v)_K + (cu,v)\forany u,v\in \NC_0^h+H_0^1(\O),
\]
and $\ell: \NC_0^h+H_0^1(\O) \to \mbR$ is as in \eqref{eq:a-def}.

\begin{theorem}\label{thm:conv}
  Let $u\in H^1_0(\O)\cap H^2(\O)$ and $u_h\in \NC_0^h$ be the solutions
  of \eqref{eq:diri-weak} and \eqref{eq:nc-ell}, respectively.
Then the following optimal error estimates hold for
the second--order elliptic problems:
\begin{subeqnarray}
\| u_h - u \|_{1,h} &\le& C h  | u |_{2},
\\
\| u_h - u \|_0 &\le& C h^2  | u |_{2}.
\end{subeqnarray}
\end{theorem}
\begin{proof}
The theorem follows from the usual argument by using the second Strang
lemma
and the interpolation estimate \eqref{eq:interpol}.
\end{proof}

\begin{acknowledgement}
The research was supported in part by National Research Foundation of
Korea (NRF--2017R1A2B3012506 and NRF--2015M3C4A7065662).
The author wishes to express his thanks to anonymous referees whose
critical comments lead to improve the manuscript substantially.
\end{acknowledgement}

\bibliographystyle{spmpsci}

\end{document}